\documentclass[10pt]{amsart}\oddsidemargin -1mm \evensidemargin -1mm \topmargin -15mm\headheight5mm \headsep 3.5mm \textheight 245mm\textwidth 170mm 
\usepackage{amsfonts, amssymb, amsthm, amsmath, euscript, hhline}
\theoremstyle{definition}\newtheorem{theorem}{Theorem}[section]\newtheorem{lemma}[theorem]{Lemma}\newtheorem{remark}[theorem]{Remark}\newtheorem{proposition}[theorem]{Proposition}\newtheorem{corollary}[theorem]{Corollary}\newtheorem{definition}[theorem]{Definition}\newtheorem{Ex}[theorem]{Example}
\DeclareMathOperator{\Ue}{\mathrm{U}}\DeclareMathOperator{\gr}{\mathrm{gr}}\DeclareMathOperator{\Sa}{\mathrm{S}}\DeclareMathOperator{\Var}{\mathrm{Var}}\DeclareMathOperator{\Ann}{\mathrm{Ann}}

\newcounter{AP}
\begin{document}
\author{Alexey Petukhov}\address{Alexey Petukhov:
The University of Manchester, Oxford Road M13 9PL, Manchester, UK\\ on leave from the Institute for Information Transmission Problems, Bolshoy Karetniy 19-1, Moscow 127994, Russia}
\email{alex-{}-2@yandex.ru}
\title{Finite-dimensional representations of minimal nilpotent W-algebras and zigzag algebras}\maketitle
\begin{abstract}Let $\frak g$ be a simple finite-dimensional Lie algebra over an algebraically closed field $\mathbb F$ of characteristic 0. We denote by $\Ue(\frak g)$ the universal enveloping algebra of $\frak g$. To any nilpotent element $e\in \frak g$ one can attach an associative (and noncommutative as a  general rule) algebra $\Ue(\frak g, e)$ which is in a proper sense a ``tensor factor'' of $\Ue(\frak g)$. In this article we consider the case in which $\frak g$ is simple and $e$ belongs of the minimal nonzero nilpotent orbit of $\frak g$. Under these assumptions $\Ue(\frak g, e)$ was described explicitly in terms of generators and relations. One can expect that the representation theory of $\Ue(\frak g, e)$ would be very similar to the representation theory of $\Ue(\frak g)$. For example one can guess that the category of finite-dimensional $\Ue(\frak g, e)$-modules is semisimple.

The goal of this article is to show that this is the case if $\frak g$ is not simply-laced. We also show that, if $\frak g$ is simply-laced and is not of type $A_n$, then the regular block of finite-dimensional $\operatorname{U}(\frak g, e)$-modules is equivalent to the category of finite-dimensional modules of a zigzag algebra.\end{abstract}
\section{Introduction}
Let $\frak g$ be a simple finite-dimensional Lie algebra over an algebraically closed field $\mathbb F$ of characteristic 0. We denote by $\Ue(\frak g)$ the universal enveloping algebra of $\frak g$. To any nilpotent element $e\in \frak g$ one can attach an associative (and noncommutative as a  general rule) algebra $\Ue(\frak g, e)$ which is in a proper sense a ``tensor factor'' of $\Ue(\frak g)$, see~\cite[Theorem~1.2.1]{LosQ}, \cite[Theorem~2.1]{Pet}. The notion of W-algebra can be traced back to the work~\cite{Lyn}, see also~\cite{Kos}. The modern definition of it (which is valid for all nilpotent elements) was given by A.~Premet~\cite{Pr02}. It turns out that the simple finite-dimensional modules of W-algebras are closely related to the primitive ideals of $\Ue(\frak g, e)$, see~\cite[Conjecture~1.2.1]{Los}. The simple finite-dimensional modules of W-algebras attract a considerable attention in the last decade, see~\cite{BG, Br,  BK,  Do, LO, Los,  PT}.

In this article we focus on the case in which $e$ belongs to the minimal nonzero nilpotent orbit of $\frak g$. Under these assumptions $\Ue(\frak g, e)$ was described explicitly in terms of generators and relations in~\cite[Theorem~1.1]{Pr1}. Moreover, a gap between primitive ideals of $\Ue(\frak g)$ and primitive ideals of $\Ue(\frak g, e)$ is very small, see~\cite[Theorem~5.3]{Pr1}. One can guess that the representation theory of $\Ue(\frak g, e)$ would be very similar to the representation theory of $\Ue(\frak g)$. For example one can guess that the category of finite-dimensional $\Ue(\frak g, e)$-modules is semisimple.

We will show that, under the assumption that $\frak g\not\cong\frak{sl}(n)$, this is true if and only if $\frak g$ is not simply-laced, see Theorems~\ref{T1},~\ref{T1nint}. Moreover, we show that, if $\frak g$ is simple simply-laced and the Dynkin diagram $\Gamma$ of $\frak g$ is not of type $A_n$, then the regular block of the category of finite-dimensional modules of $\Ue(\frak g, e)$ is equivalent to the category of finite-dimensional representations of a zigzag algebra $A(\Gamma)$ which was introduced by R.~Huerfano and M.~Khovanov~\cite{HK} (they also provided a description of indecomposable modules of $A(\Gamma)$). Explicit generators and relations for these algebras are presented in the statement of Theorem~\ref{Talg}.

The cases of $\Gamma$ of types $C_n$ and $G_2$ were considered in~\cite[Corollary~7.1]{Pr1}. Therefore the semisimplicity result is new if $\frak g$ is of type $B_n$ and $F_4$. If $\frak g\cong\frak{sl}(2)$, then $\Ue(\frak g, e)$ is isomorphic to an algebra of polynomials in one variable. The regular block of the category of finite-dimensional representations of this polynomial algebra is quite reasonable but it has no enough projective objects. Therefore this block is not equivalent to the category of finite-dimensional modules over a finite-dimensional algebra.  One can expect a similar situation if $\frak g\cong\frak{sl}(n)~(n>2)$. Nevertheless, it is plausible that the category of finite-dimensional $\Ue(\frak g, e)$-modules in this case can be described as locally nilpotent modules over a quiver with relations, see~\cite[Theorem~1.1]{GS} for an example of such a category. Perhaps, this can be done through a reduction to a similar question on a proper version of a proper Hecke algebra, see~\cite[Theorem~A]{BK}.




The paper is organized as follows. In Section~\ref{Snmr} we state the main results, i.e. Theorems~\ref{T1},~\ref{T1nint}, and introduce the notation which we need to do this. In Section~\ref{Suge} we recall several facts on W-algebras. In Section~\ref{Scsfd} we introduce the standard notation related to the simple Lie algebras and recall the notion of a cell in a reflection group. In Section~\ref{Sprmin} we study primitive ideals attached to the minimal nonzero nilpotent orbit of $\frak g$. In Section~\ref{Sprf} we recall the notion of a projective functor and several properties of it. In Section~\ref{Spr} we prove Theorem~\ref{T1}. In Section~\ref{Snint} we prove Theorem~\ref{T1nint}. In Appendix we write down a numerical result on maximal subalgebras which is needed in Section~\ref{Snint}.
\section{Notation and the main result}\label{Snmr}
Let $\frak g$ be a simple finite-dimensional Lie algebra over an algebraically closed field $\mathbb F$ of characteristic 0, $\Gamma$ be a Dynkin diagram of $\frak g$, and $e\in\frak g$ be an element of the minimal nonzero nilpotent orbit $O\subset\frak g$. We recall that to a pair $(\frak g, e)$ one can attach an associative algebra $\Ue(\frak g, e)$, see~\cite{Pr1} for details. For an $\mathbb F$-vector space $V$ we denote by $\dim V$ the dimension $V$. For an ideal $I$ of an algebra $A$ we denote by $\operatorname{Dim} I$ the Gelfand-Kirillov dimension of $A/I$. For a set $S$ we denote by $|S|$ the number of elements in $S$.

We denote by $\operatorname{Z}(\frak g)$ the center of $\Ue(\frak g)$. Algebra $\operatorname{Z}(\frak g)$ can be canonically identified with the center of $\Ue(\frak g, e)$, see~\cite[Corollary~5.1]{Pr1}, see~\cite[footnote~2]{Pr1} for the general case, and thus we also use notation $\operatorname{Z}(\frak g)$ to denote the center of $\Ue(\frak g, e)$. Let $m_0$ be the intersection of the augmentation ideal $(\frak g)$ of $\Ue(\frak g)$ with $\operatorname{Z}(\frak g)$. It is clear that $m_0$ is a maximal ideal of $\operatorname{Z}(\frak g)$. For any maximal ideal $m$ of $\operatorname{Z}(\frak g)$ we denote by
$$\Ue(\frak g, e)-f.d.mod^{m}$$
the category of finite-dimensional $\Ue(\frak g, e)$-modules $M$ such that $$\forall x\in M\exists d\in\mathbb Z_{>0}(m^dx=0).$$

Our main results for the simply-laced Lie algebras is as follows.
\begin{theorem}\label{T1} Assume that $\frak g$ is simply-laced and that $\frak g$ is not of type $A_n$. Then

a) $\Ue(\frak g, e)-f.d.mod^{m_0}$ contains exactly rank $\frak g$ simple objects,

b) $\Ue(\frak g, e)-f.d.mod^{m_0}$ is equivalent to the category of representations of the zigzag algebra $A(\Gamma)$, see~\cite{HK}.\end{theorem}
The result of Theorem~\ref{T1} can be enhanced by the following proposition.
\begin{proposition}\label{Pt1} Assume that $\frak g$ is simply-laced and not of type $A_n$. If $\lambda$ is nonintegral then the category $\operatorname{U}(\frak g, e)-f.d.mod^{m_\lambda}$ contains no nonzero objects.
\end{proposition}
The main our result for the non-simply-laced Lie algebras is as follows.
\begin{theorem}\label{T1nint} Assume that $\frak g$ is not simply-laced and is simple. Then $\operatorname{U}(\frak g, e)-f.d.mod$ is semisimple.\end{theorem}
\begin{remark} Under the assumptions of Theorem~\ref{T1} the category $\Ue(\frak g, e)-f.d.mod^{m_0}$ has finitely many  isomorphism classes of its indecomposable objects which are parametrised by the roots of $\frak g$, see~\cite[Corollary~1]{HK}. In his talk on the conference ``Representation theory and symplectic singularities'' in Edinburgh T.~Arakawa provided a connection between these W-algebras and a proper class of vertex algebras. Being motivated by this connection he expressed a hope that these categories $\Ue(\frak g, e)-f.d.mod$ would be semisimple. It turns out that this is exactly the case if $\frak g$ is not simply-laced, and if $\frak g$ is simply-laced the situation is almost as good as he expected.\end{remark}
\begin{remark} We wish to mention that $\operatorname{U}(\frak g, e)-f.d.mod$ is equivalent to a subcategory of the category of $\frak g$-modules, see Section~\ref{Spr}. This allows one to apply to $\operatorname{U}(\frak g, e)-f.d.mod$ the technique of translation functors developed in~\cite{BJ}, see also~\cite{BG}. Under the assumption that $\frak g$ is simly-laced this shows that any block of $\operatorname{U}(\frak g, e)-f.d.mod$ is either semisimple with a unique simple object or is equivalent to the category of representations of the zigzag algebra $A(\Gamma)$.\end{remark}

\begin{remark}The dimensions of the simple finite-dimensional $\operatorname{U}(\frak g, e)$-modules can be computed through the Goldie ranks of primitive ideals of $\operatorname{U}(\frak g)$, see~\cite[Theorem~5.3(2)]{Pr1}. If $\frak g$ is not simply-laced then this fact leads to a very explicit answer, see~\cite[Theorem~6.2]{Pr1}.\end{remark}

To prove Theorem~\ref{T1} we use a connection between simple $\Ue(\frak g, e)$-modules and primitive ideals of $\Ue(\frak g)$, see~\cite[Conjecture~1.2.1]{Los}. Using this approach, one can classify all simple finite-dimensional $\Ue(\frak g, e)$-modules, see~\cite{LO}.

\section{Properties of $\Ue(\frak g, e)$}\label{Suge}
Let $e\in\frak g$ be a nilpotent element. For a general definition of W-algebra $\Ue(\frak g, e)$ see~\cite{Pr1}. Here we explore the features of this object which we need in this work.

\subsection{Skryabin's equivalence} To an element $e\in\frak g$ one can assign (in a noncanonical way) a Lie subalgebra $$\frak m(e)\subset(\frak  g\oplus\mathbb F)\subset\frak \Ue(\frak g)$$ such that the category $\Ue(\frak g, e)-mod$ of $\Ue(\frak g, e)$-modules is equivalent to the category $(\frak g, \frak m(e))-l.n.mod$ of $\frak g$-modules with a locally nilpotent action of $\frak m(e)$ (Skryabin's equivalence, see~\cite{Pr1} or \cite{Los} for details). We use a particular choice of $\frak m(e)$ defined by~\cite[Subsection~2.1]{Pr1}. For a $\Ue(\frak g, e)$-module $M$ we denote by $\operatorname{Skr}(M)$ the corresponding $(\frak g, \frak m(e))$-module. This immediately defines a map
$$\mathcal P: M\to\Ann_{\Ue(\frak g)}\operatorname{Skr}(M)$$
from the set of simple finite-dimensional $\Ue(\frak g, e)$-modules to the set of primitive ideals of $\Ue(\frak g)$. The following proposition describes the image of $\mathcal P$ under the assumption that $O$ is the minimal nilpotent orbit of $\frak g$ (for the general case see~\cite{LO}).
\begin{proposition}\label{Pprw} Assume that $O$ is the minimal nonzero nilpotent orbit of $\frak g$. Let $m$ be a maximal ideal of $\operatorname{Z}(\frak g)$. Then $\mathcal P$ defines a bijection between the isomorphism classes of finite-dimensional simple $\Ue(\frak g, e)$-modules which are annihilated by $m$, and the
primitive ideals $I$ of $\operatorname{U}(\frak g)$ with $I\supset m$ and $\operatorname{Var}(I) = \bar{O}$ where $\operatorname{Var}(I)$ is the associated variety of $I$ defined in~\cite[Subsection~3.2]{Pr1}, see also~Subsection~\ref{SSav}.\end{proposition}
\begin{proof}Is equivalent to~\cite[Theorem~5.3(5)]{Pr1}.\end{proof}
\subsection{Generators and relations for $\Ue(\frak g, e)$.}\label{SSger} Denote by $\frak g_e$ the centralizer of $e$ in $\frak g$. One can consider $\Ue(\frak g, e)$ as a deformation of the universal enveloping algebra $\Ue(\frak g_e)$. Using this approach one can provide $\Ue(\frak g, e)$ with a PBW-basis and evaluate the defining set of relations, see~\cite[Subsection~1.1]{Pr1}. These generators and relations are known explicitly under the assumption that $e$ belongs to the minimal nilpotent orbit $O$ as we explain next (see also~\cite[Theorem~1.1]{Pr1}).

From now on $e\in O\subset \frak g$. The associative algebra $\Ue(\frak g, e)$ is generated by the subspaces $\frak g_e(0)$ and $\frak g_e(1)$ and the central element $C$ modulo the following relations:

1) $\forall x, y\in\frak g_e(0)~([x, y]=xy-yx\in \frak g_e(0))$ and thus $\frak g_e(0)$ is a Lie algebra,

2) $\forall x\in\frak g_e(0)\forall y\in\frak g_e(1)~([x, y]=xy-yx\in \frak g_e(1))$, i.e. $\frak g_e(1)$ is a $\frak g_e(0)$-module,

3) a formula for $[x, y]~(x, y\in\frak g_e(1))$, see~\cite[Theorem~1.1]{Pr1}.\\
Using these formulas one can easily check that, if $\frak g$ is not of type A, then $\Ue(\frak g, e)$ has a unique one-dimensional module which is isomorphic to
$$\Ue(\frak g, e)/\Ue(\frak g, e)(\frak g_e(0)\oplus\frak g_e(1)),$$
see~\cite[Corollary~4.1]{Pr1}.
The following proposition is crucial for the present work.
\begin{proposition}\label{Pext}Assume that $\frak g$ is not of type $A$ and that $M$ is the one-dimensional $\Ue(\frak g, e)$-module defined above. Then $\operatorname{Ext}^1_{\Ue(\frak g, e)}(M, M)=0$.\end{proposition}
\begin{proof}It is enough to show that $M$ has no nontrivial self-extensions. Indeed, let$$0\to M\to \tilde{M}\to M\to0$$be a self-extension of $M$. Then $$(\Ann_{\Ue(\frak g, e)}M)^2\subset \Ann_{\Ue(\frak g, e)}\tilde{M}\subset\Ann_{\Ue(\frak g, e)}M.$$
We claim that $\Ann_{\Ue(\frak g)}\tilde{M}=\Ann_{\Ue(\frak g, e)}M$. To show this we prove that $\Ann_{\Ue(\frak g, e)}M=(\Ann_{\Ue(\frak g, e)}M)^2$.

One has that $[\frak g_e(0), \frak g_e(0)]=\frak g_e(0)$ and $[\frak g_e(0), \frak g_e(1)]=\frak g_e(1)$, see~\cite[Corollary~4.1]{Pr1}. This implies that $\Ann_{\Ue(\frak g, e)}M$ is generated by $\frak g_e(0)$ as a two-sided ideal. Using once more that $[\frak g_e(0), \frak g_e(0)]=\frak g_e(0)$ we see that $\Ann_{\Ue(\frak g, e)}M=(\Ann_{\Ue(\frak g, e)}M)^2$.

The claim implies that $\Ann_{\Ue(\frak g, e)}\tilde{M}$ is a $(\Ue(\frak g, e)/\Ann_{\Ue(\frak g, e)}M)$-module. The fact that $M$ is one-dimensional implies that $\Ue(\frak g, e)/\Ann_{\Ue(\frak g, e)}M\cong\mathbb F$. Therefore $\tilde{M}\cong M\oplus M$. Thus $\operatorname{Ext}^1_{\Ue(\frak g, e)}(M, M)=0$. 
\end{proof}
\section{On the classification of primitive ideals of $\Ue(\frak g)$.}\label{Scsfd}
We need a quite detailed description of the set of primitive ideals of $\Ue(\frak g)$ together with the respective notation.

\subsection{Notation}We assume that $\frak g$ is a simple Lie algebra. Denote by $\frak b\subset\frak g$ a Borel subalgebra of $\frak g$ and by $\frak h\subset\frak b$ a Cartan subalgebra of $\frak b$. We have $$\frak g=\oplus_{\alpha\in\frak h^*}\frak g_\alpha$$
where $\frak g_0=\frak h$, $\dim\frak g_\alpha\le 1$ if $\alpha\ne 0$, and if $\frak g_\alpha\ne 0$ then $\frak g_\alpha$ is a simple one-dimensional $\frak h$-module with character $\alpha$. We put $$\Delta:=\{0\ne\alpha\in\frak h^*\mid\frak g_\alpha\ne  0\},\hspace{10pt}\Delta^+:=\{0\ne\alpha\in\frak h^*\mid \frak g_\alpha\subset\frak b\},\hspace{10pt}\rho:=\frac12\sum_{\alpha\in\Delta^+}\alpha.$$
We denote

$\bullet$ by $(\cdot, \cdot)$ the Cartan-Kiling form of $\frak g$,

$\bullet$ by $\Pi$ the simple roots of $\Delta^+$, and by $$\{\omega(\alpha)\}_{\alpha\in\Pi}\subset\frak h^*$$ the corresponding fundamental weights,

$\bullet$ by $\Lambda$ the lattice generated by $\{\omega(\alpha)\}_{\alpha\in\Pi}$,

$\bullet$ by $\Lambda^+$ the semigroup with 0 generated by $\{\omega(\alpha)\}_{\alpha\in\Pi}$,

$\bullet$ by $W$ the subgroup generated by the reflections with respect to the elements of $\Delta$.\\
Note that $(\cdot, \cdot)$ canonically identifies $\frak g$ and $\frak g^*$ and is nondegenerate after the restriction to $\frak h$. Hence it also identifies $\frak h$ and $\frak h^*$. 

Fix $\lambda\in\Lambda$. Put $$\Delta_\lambda:=\{\alpha\in\Delta\mid \frac{2(\lambda+\rho, \alpha)}{(\alpha, \alpha)}=0\}.$$

\begin{definition}We say that $\lambda$ is {\it singular} if $\Delta_{\lambda}\ne \emptyset$, and we say that $\lambda$ is {\it regular} otherwise.\end{definition}
\begin{definition}We say that $\lambda\in\Lambda$ is {\it dominant} if $\lambda\in\Lambda^+$. We say that $\lambda$ is {\it $\rho$-dominant} if $\lambda+\rho$ is dominant.\end{definition}
\begin{definition}We say that two roots $\alpha, \beta\in\Pi$ are {\it adjacent} if $\alpha\ne\beta$ and $(\alpha, \beta)\ne 0$.\end{definition}

To any $\lambda\in\frak h^*$ we assign a one-dimensional  $\frak h$-module $\mathbb F_\lambda$ which we also consider as a $\frak b$-module ($\frak h\cong\frak b/\frak n$ where $\frak n$ is the nilpotent radical of $\frak b$). Put
$$M(\lambda):=\Ue(\frak g)\otimes_{\Ue(\frak b)}\mathbb F_\lambda,\hspace{10pt}m_\lambda:=\operatorname{Z}(\frak g)\cap\Ann_{\Ue(\frak g)}M(\lambda).$$
The ideal $m_\lambda$ is maximal in $\operatorname{Z}(\frak g)$ for all $\lambda$ and thus we have a map from $\frak h^*$ to the set of maximal ideals of $\operatorname{Z}(\frak g)$. Moreover, $$m_\lambda=m_\mu\Leftrightarrow\exists w\in W(w(\lambda+\rho)=\mu+\rho).$$

For any $\alpha\in\Delta$ we denote by $s_{\alpha}$ the corresponding reflection. For any $w\in W$ we put$$l(w):=|-\Delta^+\cap w\Delta^+|.$$

We set $L(\lambda)$ to be a unique simple quotient of $M(\lambda)$ and $I(\lambda):=\Ann_{\Ue(\frak g)}L(\lambda)$. According to Duflo's theorem, for any primitive ideal $I$ of $\Ue(\frak g)$ there exists $\lambda\in\frak h^*$ such that $I=I(\lambda)$.  Put $$\tau_L(w):=\{\alpha\in \Pi\mid  w^{-1}\alpha\in\Delta^+\},\hspace{10pt}\tau_R(w):=\{\alpha\in \Pi\mid w\alpha\in\Delta^+\}.$$
The following lemma provides a very useful invariant of primitive ideals.
\begin{lemma} Fix $w_1, w_2\in W$. If $I(w_1\rho-\rho)=I(w_2\rho-\rho)$ then $\tau_R(w_1)=\tau_R(w_2)$. Thus we can define $$\tau(I):=\tau_R(w_1)=\tau_R(w_2).$$\end{lemma}
\begin{proof}See~\cite[Subsection~2.14]{BJ}, see also~\cite[Theorem~2.4]{V} and the text above it.\end{proof}
\subsection{Associated varieties of ideals}\label{SSav}
The universal enveloping algebra $\Ue(\frak g)$ of $\frak g$ has natural degree filtration $\{U_i\}_{i\ge0}$. The associated graded algebra $$\operatorname{gr}\Ue(\frak g):=\oplus_{i\ge 0}(U_i/U_{i-1})$$ is canonically isomorphic to the symmetric algebra $\Sa^\cdot(\frak g)$ of $\frak g$. For a two-sided ideal $I$ we put $$\operatorname{gr} I:=\oplus_{i\ge0}(I\cap U_i/I\cap I_{i-1})\subset\operatorname{gr}\Ue(\frak g)\cong \Sa^\cdot(\frak g).$$We put $\Var(I):=\{x\in\frak g^*\mid \forall f\in \gr I~(f(x)=0)\}$. It is known that if $I$ is a primitive ideal of $\Ue(\frak g)$ then $\Var(I)$ is the closure $\overline{O(I)}$ of a nilpotent coadjoint orbit $O(I)$ in $\frak g^*$, see~\cite{JoA}. We put $O(\lambda):=O(I(\lambda))$.

\subsection{One-sided and two-sided cells}\label{SScell}  Next, we need some combinatorial data attached to the reflection group $W$. It has some definition through the Kazhdan-Lusztig polynomials, see, for example,~\cite[Section~6]{LO}, and it has a much more explicit description for all classical Lie algebras, see~\cite{BV}. But we believe that the approach which we present here also makes sense, cf. with~\cite{BB}. We introduce two relations on $W$:

1) $w_1\sim_L w_2\Leftrightarrow I(w_1\rho-\rho)=I(w_2\rho-\rho)~(w_1, w_2\in W)$,

2)  $w_1\sim_R w_2\Leftrightarrow I(w_1^{-1}\rho-\rho)=I(w_2^{-1}\rho-\rho)~(w_1, w_2\in W)$.\\
Clearly, $\sim_L$ and $\sim_R$ are equivalence relations on $W$, and we denote by $\sim$ the smallest equivalence relation on $W$ which includes both $\sim_L, \sim_R$. We denote by

$\bullet$ $\operatorname{LCell}(w, \Pi)$ the equivalence class of $\sim_L$ which contains $w$,

$\bullet$ $\operatorname{RCell}(w, \Pi)$ the equivalence class of $\sim_R$ which contains $w$,

$\bullet$ $\operatorname{TCell}(w, \Pi)$ the equivalence classes of $\sim$ which contains $w$.\\
It is clear that the set of left cells can be naturally identified with the set of primitive ideal of the form $I(w\rho-\rho)~(w\in W)$. The following proposition gives a straightforward connection between the associated varieties of ideals and the two-sided cells.
\begin{proposition}[{~\cite[Subsection~6.2]{LO}}]\label{Passv} Two-sided ideals $I(w_1\rho-\rho)$ and $I(w_2\rho-\rho)$ have the same associated variety if and only if $w_1\sim w_2$.\end{proposition}Proposition~\ref{Passv} defines a map from the set of two-sided cells of $W$ to the set of nilpotent orbits. An orbit which belongs to the image of this map is called {\it special}. The minimal nilpotent orbit $O$ is special if and only if $\frak g$ is simply-laced, see~\cite{CM}.
\section{Primitive ideals $I$ for which $\operatorname{Var}(I)=\bar{O}$.}\label{Sprmin} Let $\lambda\in\Lambda$ be a weight. Denote by $\operatorname{PrId}^\lambda(O)$ the set of primitive ideals $I$ of $\Ue(\frak g)$ such that $\operatorname{Var}(I)=\bar{O}$ and $m_\lambda\subset I$.
\begin{lemma}\label{Lprr}We  have $|\operatorname{PrId}^0(O)|=|\Pi|$.\end{lemma}
\begin{proof} Denote by $\operatorname{TCell}(O)$ the two-sided cell in $W$ attached to $O$ through Proposition~\ref{Passv}. The desired number of ideals equals to the number of left cells in $\operatorname{TCell}(O)$, see~Subsection~\ref{SScell}. According to~\cite[Subsection 4.3]{JoM}, $\operatorname{TCell}(O)$ splits into exactly $|\Pi|$ left cells. \end{proof}
\begin{remark}One can deduce Lemma~\ref{Lprr} from~\cite{Dou}.\end{remark}
To work out the singular integral case we need the following proposition.
\begin{proposition}[{\cite[Satz~2.14]{BJ}}]\label{Pps}Fix $\lambda\in\Lambda$. Then the following sets can be identified

1) primitive ideals $I$ of $\Ue(\frak g)$ such that $I\supset m_\lambda$,

2) primitive ideals $I$ of $\Ue(\frak g)$ such that $I\supset m_0$ and $\tau(I)\cap \Delta_\lambda=\emptyset$.\\
The associated varieties of the ideals identical under this correspondence are the same.
\end{proposition}
Using this proposition we can evaluate the desired classification of the singular cases which we need.
\begin{proposition}\label{Ppros} a) For $I\in\operatorname{PrId}^0(O)$ there exists $\alpha\in\Pi$ such that $\tau(I)=\Pi\backslash\alpha$.

b) For any $\alpha\in\Pi$ there exists and unique $I\in\operatorname{PrId}^0(O)$ such that $\tau(I)=\Pi\backslash\alpha$.

c) For any $\alpha\in\Pi$ there exists exactly one primitive ideal $I$ such that $I\supset m_{-\omega(\alpha)}$.
\end{proposition}
\begin{proof} Part c) is implied  by parts a) and b), and Proposition~\ref{Pps}. Part a) is implied by part b) and Lemma~\ref{Lprr}. To prove part b) we first show that for any $\alpha\in\Pi$ there exists $I\in\operatorname{PrId}^0(O)$ such that $\tau(I)=\Pi\backslash\alpha$.

Denote by $\alpha_0$ the unique simple root with 3 neighbours. For any $\alpha\in\Pi$ put $\alpha_0, \alpha_1,..., \alpha_{n(\alpha)}$ to be the shortest sequence of adjacent roots which connects $\alpha_0$ with $\alpha=\alpha_{n(\alpha)}$. Set $$w_\alpha:=s_{n(\alpha)}...s_{\alpha_1}s_{\alpha_0}\in W$$ where $s_{\alpha_i}\in W$ is the reflection with respect to $\alpha_i$. It follows from~\cite[Subsection 4.3]{JoM} that
$$\{w_{\alpha}\}_{\alpha\in\Pi}$$ is a left cell in $\operatorname{TCell}(O)$. Thus

$$\{w_{\alpha}^{-1}\}_{\alpha\in\Pi}$$ is a right cell in $\operatorname{TCell}(O)$. One can check that $$\tau(I(w_{\alpha}^{-1}\rho-\rho))=\tau_R(w_\alpha^{-1})=\tau_L(w_\alpha)=\Pi\backslash\{\alpha\},$$see also~\cite[Subsection 4.3-4.4]{JoM}.

To complete the proof of part b) we mention that the ideals $I(w_\alpha^{-1}\rho-\rho)~(\alpha\in\Pi)$ are distinct, and that according to Lemma~\ref{Lprr} we have $|\operatorname{PrId}^0(O)|=|\Pi|$.\end{proof}
\begin{proof}[Proof of Theorem~\ref{T1}a)] Is a composition of Propositions~\ref{Pprw},~\ref{Ppros}.\end{proof}
\section{The semiring of projective functors}\label{Sprf} Let $\lambda\in\Lambda$ be a weight and $V$ be a finite-dimensional $\frak g$-module. We say that a functor $$F: \Ue(\frak g)-mod^{m_\lambda}\to \Ue(\frak g)-mod$$ is {\it projective} if it is a direct summand of a functor $$\cdot\otimes V: M\to M\otimes V~(\Ue(\frak g)-mod^{m_\lambda}\to \Ue(\frak g)-mod),$$
cf. with~\cite{BGG}. It can be checked that $\cdot\otimes V$ maps $\Ue(\frak g)-mod^{m_\lambda}$ to $\oplus_{\mu+\rho\in\Lambda^+}\Ue(\frak g)-mod^{m_{\mu}}\subset\Ue(\frak g)-mod$ and therefore one can consider projective functors as endofunctors of the category $\oplus_{\mu+\rho\in\Lambda^+}\Ue(\frak g)-mod^{m_{\mu}}$.

The projective functors enjoy the following properties:

$\bullet$ a projective functor is exact,

$\bullet$ a projective functor is a direct sum of finitely many indecomposable projective functors,

$\bullet$ the composition of projective functors is projective,

$\bullet$ the direct sum of projective functors is projective,\\
see~\cite[Section~3]{BGG}. The indecomposable projective functors can be described as follows. For any pair $\chi, \xi\in\Lambda$ one can assign an indecomposable projective functor $$F_{\chi, \xi}: \Ue(\frak g)-mod^{m_\chi}\to \Ue(\frak g)-mod^{m_\xi}.$$
Two such functors $F_{\chi_1, \xi_1}$ and $F_{\chi_2, \xi_2}$ are isomorphic if and only if there exists $w\in W$ such that \begin{center}$w(\chi_1+\rho)=\chi_2+\rho$ and $w(\xi_1+\rho)=\xi_2+\rho$.\end{center}

We would be particularly interested in the following collection of functors:
$$\psi_\alpha:=F_{0, -\omega(\alpha)}:~\Ue(\frak g)-mod^{m_0}\to\Ue(\frak g)-mod^{m_{-\omega(\alpha)}}~(\alpha\in\Pi),$$
$$\phi_\alpha:=F_{-\omega(\alpha), 0}:~\Ue(\frak g)-mod^{m_{-\omega(\alpha)}}\to\Ue(\frak g)-mod^{m_0}~(\alpha\in\Pi),$$
$$T_\alpha:=F_{0, s_\alpha\rho-\rho}: \Ue(\frak g)-mod^{m_0}\to\Ue(\frak g)-mod^{m_{s_\alpha\rho-\rho}}=\Ue(\frak g)-mod^{m_0}~(\alpha\in \Pi).$$

It is clear that the indecomposable projective functors form a semiring with respect to the direct sum (considered as an addition) and the composition (considered as a multiplication). We denote this semiring $\mathcal R$. The following lemma is quite standard and is pretty straightforward.
\begin{lemma}Let $Ring_0(\mathcal R)$ be the set of pairs $(r_1, r_2)\in\mathcal R\times\mathcal R$ modulo the equivalence relation $(a, b)\sim(a', b')$ if and only if there exists $t\in\mathcal R$ such that $a+b'+t=a'+b+t$.

a) The operations $(r_1, r_2)+(r_3, r_4):=(r_1+r_3, r_2+r_4)$, $(r_1, r_2)\cdot(r_3, r_4):=(r_1r_3+r_2r_4, r_1r_4+r_2r_3)$ define a structure of a ring on $Ring_0(\mathcal R)$. We denote this ring $Ring(\mathcal R)$. We denote by $$\phi_\mathcal R: \mathcal R\to Ring(\mathcal R)~(r\to (r+r, r))$$ the respective morphism of semirings.

b) If $R'$ is a ring and $\phi'$ is a morphism of semirings then there exists and unique morphism of rings $$\psi: Ring(\mathcal R)\to R'$$ such that $\phi':=\psi\circ \phi$.\end{lemma}

Next, we note that $\mathcal R$ naturally acts on the Grothendieck $K$-group of the category of $\frak g$-modules (this is a straightforward check through the definitions of an exact functor and a $K$-group). Moreover, if $\mathcal C$ is a subcategory of $$\oplus_{\mu+\rho\in\Lambda^+}\Ue(\frak g)-mod^{m_\mu}$$ which is stable under $\cdot\otimes V$ for any finite-dimensional $\frak g$-module $V$, then $\mathcal R$ acts on the $K$-group $K(\mathcal C)$ of $\mathcal C$. The endomorphisms of $K(\mathcal C)$ form  a ring and hence we have a natural action of $Ring(\mathcal R)$ on $K(\mathcal C)$.
\subsection{Basis-dependent description of $Ring(\mathcal R)$} We recall that one can attach to a projective functor $F_{\chi, \xi}$ an endomorphism $F_{\chi, \xi}^K$ of a free lattice generated by $\{\delta_\lambda\}_{\lambda\in\Lambda}$, see~\cite[Subsection~3.4]{BGG} (this corresponds to the action of $Ring(\mathcal R)$ on the Grothendieck group $K(\mathcal O)$ of category $\mathcal O$). The assignment enjoy the following properties

$\bullet$ $F_{\chi_1, \xi_1}^K=F_{\chi_2, \xi_2}^K$ implies that $F_{\chi_1, \xi_1}\cong F_{\chi_2, \xi_2}$,

$\bullet$ $(F_{\chi, \xi})^K(\delta_\lambda)\ne 0$ if and only if $m_\chi=m_\lambda$,

$\bullet$ $F_{\chi_1, \xi_1}^K+F_{\chi_2, \xi_2}^K=(F_{\chi_1, \xi_1}\oplus F_{\chi_2, \xi_2})^K$,

$\bullet$ $F_{\chi_1, \xi_1}^KF_{\chi_2, \xi_2}^K=(F_{\chi_1, \xi_1}\circ F_{\chi_2, \xi_2})^K$\\
($\chi_1, \chi_2, \xi_1, \xi_2, \lambda\in\Lambda$), see~\cite[Subsection~3.4]{BGG}. This immediately implies that the map $(\cdot)^K$ defines the morphism from $Ring(\mathcal R)$ to the endomorphisms of the lattice  $\oplus_{\lambda\in\Lambda}\mathbb F\delta_\lambda$ and that this map is injective. 

Next, we mention that the lattice $\oplus_{\lambda\in\Lambda}\mathbb Z\delta_\lambda$ carries a $W$-action defined by the formula
$$w\cdot\delta_{\lambda}:=\delta_{w(\lambda+\rho)-\rho}~(w\in W, \lambda\in\Lambda)$$
and $F^K_{\chi, \xi}$ commutes with the action of $W$. One can use this fact to provide an action of a Weyl group on the Grothendieck group of the blocks of $\frak g$-modules, see~\cite[Theorem~C.2~of~Appendix]{KZ}.
\begin{lemma} We have

a) $(\phi_{\alpha})^K(\delta_{-\omega(\alpha)})=\delta_0+\delta_{s_\alpha\rho-\rho}=\delta_0+\delta_{-\alpha}$,

b) $(\psi_{\alpha})^K(\delta_0)=\delta_{-\omega(\alpha)}.$

c) $(T_\alpha)^K(\delta_0)=\delta_0+\delta_{-\alpha}$.
\end{lemma}
\begin{proof} Part a) follows from a combination of~\cite[Subsections~1.12, 3.3, 3.4]{BGG} and~\cite[Theorem~7.14(a)]{H}. Part b) is implied by~\cite[3.3, Theorem(ii)b)]{BGG}. Parts c) is a consequence of parts a) and b).\end{proof}
For $\lambda\in\Lambda$ denote by $Id_\lambda$ the operator on $\oplus_{\mu\in\Lambda}\delta_\mu$ defined by the formula $$Id_\lambda(\delta_\mu):=\begin{cases}\delta_\mu\mbox{~if~}\mu+\rho=w(\lambda+\rho)\mbox{~for~some~}w\in W,\\0\mbox{~otherwise.}\end{cases}$$
The following corollary would be very useful in the proof of Theorem~\ref{T1}.
\begin{corollary}\label{Cf}a) $(\psi_\alpha)^K(\phi_\alpha)^K=2 Id_{-\omega(\alpha)}$,

b) $(\phi_\alpha)^K(\psi_\alpha)^K=(T_\alpha)^K$,

c) $((T_\alpha)^K)^2=2(T_\alpha)^K$, $((T_\alpha)^K-Id_0)^2=Id_0$,

d) $(((T_\alpha)^K-Id_0)((T_\beta)^K-Id_0))^3=Id_0$ if $\alpha$ and $\beta$ are adjacent,

e) $(((T_\alpha)^K-Id_0)((T_\beta)^K-Id_0))^2=Id_0$ if $\alpha$ and $\beta$ are not adjacent and $\alpha\ne\beta$.
\end{corollary}
\section{Proof of Theorem~\ref{T1}}\label{Spr}
To start with we note that $\Ue(\frak g, e)-f.d.mod$ is equivalent to the category $\mathcal C$ of $(\frak g, \frak m(e))$-l.n.modules of finite-length and of Gelfand-Kirillov dimension $\frac12\dim O$, see~\cite[Proposition 3.3.5]{LosQ}. We have$$\mathcal C\subset\oplus_{\mu+\rho\in\Lambda^+}\Ue(\frak g)-mod^{m_\mu}.$$ This implies that $Ring(\mathcal R)$ acts on $K(\mathcal C)$. For any $F\in Ring(\mathcal R)$ we denote  by $(F)^{K(\mathcal C)}$ the image of it in $\operatorname{End}(K(\mathcal C))$.

For any $\lambda\in\Lambda$ we put $\mathcal C^\lambda:=\mathcal C\cap\Ue(\frak g)-mod^{m_\lambda}$. Functors $T_\alpha, \psi_\alpha, \phi_\alpha$ acts on these categories as follows
$$\psi_\alpha:\mathcal C^0\to\mathcal C^{-\omega(\alpha)},\hspace{10pt}\phi_\alpha: \mathcal C^{-\omega(\alpha)}\to\mathcal C^0,\hspace{10pt}T_\alpha: \mathcal C^0\to\mathcal C^0.$$
We proceed with a description of $K(\mathcal C^0)$ and $K(\mathcal C^{-\omega(\alpha)})$. For an object $M$ of $\mathcal C$ we denote by $[M]$ the image of it in $K(\mathcal C)$. All objects of $\mathcal C, \mathcal C^0$ have finite length, and
thus $K(\mathcal C)$~(respectively $K(\mathcal C^0)$) are generated by the images of the simple objects of $K(\mathcal C)$ (respectively of $K(\mathcal C^0)$).


Proposition~\ref{Pprw} together with Proposition~\ref{Ppros}c) implies that $\mathcal C^{-\omega(\alpha)}$ has a unique simple object $M_{\alpha}^-$ which corresponds to the unique primitive ideal $I$ of Gelfand-Kirillov dimension $\dim O$ such that $I\supset m_{-\omega(\alpha)}$.

Proposition~\ref{Pprw} together with Proposition~\ref{Ppros} defines a bijection between simple objects of $\mathcal C^{0}$ and elements of $\Pi$. For any $\alpha\in\Pi$ we denote by $M_\alpha$ the respective simple object of $\mathcal C^0$.

We need the following lemma.
\begin{lemma}\label{Ltra}Let $F$ be a projective functor and $M_1, M_2$ be $\frak g$-modules such that $\Ann_{\Ue(\frak g)}M_1=\Ann_{\Ue(\frak g)}M_2$. Then $\Ann_{\Ue(\frak g)}F(M_1)=\Ann_{\Ue(\frak g)}F(M_2)$. In particular, $F(M_1)=0$ if and only if $F(M_2)=0$.\end{lemma}
\begin{proof}Is implied by~\cite[Lemma~2.3]{V}.\end{proof}
Assembling together Lemma~\ref{Ltra},~\cite[Satz~2.14]{BJ} and Proposition~\ref{Ppros} we prove the following lemma.
\begin{lemma}\label{Ltra-}We have $\psi_\alpha(M_\beta)=0$ if and only $\alpha\ne\beta\in\Pi$.\end{lemma}
As a corollary we have that $T_\alpha(M_\beta)=0$ if and only if $\alpha\ne\beta$; $(T_\alpha)^{K(\mathcal C)}[M_\alpha]=\Sigma_\beta c_{\alpha, \beta}[M_{\beta}]$ for some $c_{\alpha, \beta}\in\mathbb Z_{\ge0}$.
\begin{lemma}\label{Lm}We have $c_{\alpha, \alpha}=2$ for all $\alpha\in\Pi$.\end{lemma}
\begin{proof}Is implied by Lemma~\ref{Ltra-} and formula $((T_\alpha)^K)^2=2(T_\alpha)^K$.\end{proof}
\begin{lemma}\label{Lp3}a) If $\alpha, \beta\in\Pi$ are adjacent then $c_{\alpha, \beta}=c_{\beta,\alpha}=1$.

b) If $\alpha\ne\beta\in\Pi$ are not adjacent then $c_{\alpha, \beta}=c_{\beta,\alpha}=0$.
\end{lemma}
\begin{proof}Part a). We have that
\begin{equation}((T_\alpha)^K-Id_0)^{K(\mathcal C)}((T)^K_\beta-Id_0)^{K(\mathcal C)}[M_\alpha]=-[M_\alpha]-c_{\alpha, \beta}[M_\beta]-\Sigma_{\gamma\ne\alpha, \beta}c_{\alpha, \gamma}[M_\gamma],\label{Ef1}\end{equation}
\begin{equation}((T)^K_\alpha-Id_0)^{K(\mathcal C)}((T)^K_\beta-Id_0)^{K(\mathcal C)}[M_\beta]=c_{\beta, \alpha}[M_\alpha]+(c_{\beta, \alpha}c_{\alpha, \beta}-1)[M_\beta]+\Sigma_{\gamma\ne \alpha, \beta}(c_{\beta, \alpha}c_{\alpha, \gamma}-c_{\beta, \gamma})[M_\gamma],\label{Ef2}\end{equation}
\begin{equation}((T)^K_\alpha-Id_0)^{K(\mathcal C)}((T)^K_\beta-Id_0)^{K(\mathcal C)}[M_\gamma]=[M_\gamma]~(\gamma\ne\alpha, \beta).\label{Ef3}\end{equation}

We fix $\alpha, \beta\in\Pi$ such that $\alpha$ and $\beta$ are adjacent. Formulas~(\ref{Ef1}), (\ref{Ef2}), (\ref{Ef3}) together with Corollary~\ref{Cf}e) implies that
\begin{equation}(\begin{array}{cc}-1&-c_{\alpha, \beta}\\c_{\beta, \alpha}&c_{\alpha,\beta}c_{\beta,\alpha}-1\end{array})^3=(\begin{array}{cc}1&0\\0&1\end{array}).\label{Ea3}\end{equation}
We put 
$$X:=(\begin{array}{cc}-1&-c_{\alpha, \beta}\\c_{\beta, \alpha}&c_{\alpha,\beta}c_{\beta,\alpha}-1\end{array}).$$
Equation~(\ref{Ea3}) implies that all eigenvalues of $X$ are roots of unity of degree 3 and therefore the trace $\operatorname{tr}X$ of $X$ equals to the sum of two (not necessarily distinct) roots of unity of degree 3. On the other hand $\operatorname{tr}X$ equals $c_{\alpha,\beta}c_{\beta,\alpha}-2$, and hence $\operatorname{tr}X$ is an integer. It can be easily checked that if such a sum is an integer then it is equal to -1 or 2. Thus $c_{\alpha,\beta}c_{\beta,\alpha}\in\{1, 4\}$. Hence $c_{\alpha,\beta}=c_{\beta,\alpha}=1$, or $c_{\alpha,\beta}=4$ and $c_{\beta,\alpha}=1$, or $c_{\alpha,\beta}=c_{\beta,\alpha}=2$, or $c_{\alpha,\beta}=1$ and $c_{\beta,\alpha}=4$. It can be easily seen that $X^3={\bf 1}$ if and only if $c_{\alpha,\beta}=c_{\beta,\alpha}=1$.

Part b). We fix $\alpha, \beta\in\Pi$ such that $\alpha$ and $\beta$ are adjacent. Formulas~(\ref{Ef1}),~(\ref{Ef2}),~(\ref{Ef3}) together with Corollary~\ref{Cf}d) implies that $X^2={\bf 1}$. Further we have $$\operatorname{det}(X)=(-1)(c_{\alpha,\beta}c_{\beta,\alpha}-1)-(-c_{\alpha,\beta})c_{\beta,\alpha}=1.$$
These two facts together implies that $X={\bf 1}$ or $X=-{\bf 1}$. As a consequence we have $c_{\alpha,\beta}=c_{\beta,\alpha}=0$.\end{proof}
\begin{lemma}\label{Lmul} a) We have $\psi_\alpha M_\alpha\cong M_\alpha^-$.

b) If $\alpha, \beta\in\Pi$ are adjacent then we have $\psi_\beta(\phi_\alpha M_\alpha^-)\cong M_\beta^-$.\end{lemma}
\begin{proof} Part a). Category $\mathcal C^{-\omega(\alpha)}$ has the unique simple object $M_{\alpha}^-$ and thus $[\psi_\alpha M_\alpha]=c[M_{\alpha}^-]$ for some $c\in\mathbb Z_{>0}$. Next,
$$[T_\alpha M_\alpha]=[\phi_\alpha\psi_\alpha M_\alpha]=c[\phi_\alpha M_\alpha^-].$$
In particular, this implies that, for all $\beta\in\Pi$, $c_{\alpha,\beta}$ must be divisible by $c$. This together with Lemma~\ref{Lp3} implies that $c\mid 1$, and therefore that $c=1$.

Part b). It follows from part a) that $\phi_\alpha M_\alpha^-\cong T_\alpha M_\alpha$. According to Lemma~\ref{Lp3} we have that $c_{\alpha,\beta}=1$. Therefore $[\psi_\beta(\phi_\alpha M_\alpha^-)]=[M_\beta^-]$, and we have $\psi_\beta(\phi_\alpha M_\alpha^-)\cong M_\beta^-$.\end{proof}

\begin{proposition}\label{Pss}Categories $\mathcal C^{-\omega(\alpha)}$ are semisimple with a unique simple object.\end{proposition}
\begin{proof}The statement of Proposition~\ref{Pss} for $\alpha=\alpha_0$ holds thanks to Proposition~\ref{Pext}. Due to the fact that $\frak g$ is simple, it is enough to show that if the statement of Proposition~\ref{Pss} holds for $\alpha\in\Pi$ and $\beta$ is adjacent to $\alpha$, then the statement of Proposition~\ref{Pss} holds for $\beta$.

Thus we assume that $\alpha$ and $\beta$ are adjacent roots and $\mathcal C^{-\omega(\alpha)}$ is semisimple and has a unique up to isomorphism simple object $M_\alpha^-$. Thanks to Proposition~\ref{Ppros}c) and Proposition~\ref{Pprw}, it is enough to show that $M_\beta^-$ is projective in $\mathcal C^{-\omega(\beta)}$.

We have
$$\operatorname{Hom}_{\Ue(\frak g)}(M_\beta^-, M)\cong \operatorname{Hom}_{\Ue(\frak g)}(\psi_\beta\phi_\alpha M_\alpha^-, M)\cong \operatorname{Hom}_{\Ue(\frak g)}(M_\alpha^-, \psi_\alpha\phi_\beta M)$$
for an object $M$ of $\mathcal C^{-\omega(\beta)}.$ Then the fact that $\psi_\alpha, \phi_\beta$ are exact immediately implies that $M_{\beta}^-$ is projective.
\end{proof}

\begin{corollary}\label{Cmain} Put $P_\alpha:=\phi_\alpha M_\alpha^-$. For any object $M$ of $\mathcal C^0$ we have $$\dim \operatorname{Hom}_{\Ue(\frak g)}(M, P_\alpha)=\dim \operatorname{Hom}_{\Ue(\frak g)}(P_\alpha, M)$$ and equals to the Jordan-H\"older multiplicity of $M_\alpha$ in $M$. In particular, $P_\alpha$ is both projective and injective.\end{corollary}
\begin{proof} Fix an object $M$ of $\Ue(\frak g)-f.d.mod^{m_0}$. We have
$$\operatorname{Hom}_{\Ue(\frak g)}(M_\alpha^-, \psi_\alpha M)\cong \operatorname{Hom}_{\Ue(\frak g)}(\phi_\alpha M_\alpha^-, M)\cong \operatorname{Hom}_{\Ue(\frak g)}(P_\alpha, M),$$
$$\operatorname{Hom}_{\Ue(\frak g)}(\psi_\alpha M, M_\alpha^-)\cong \operatorname{Hom}_{\Ue(\frak g)}(M, \phi_\alpha M_\alpha^-)\cong \operatorname{Hom}_{\Ue(\frak g)}(M, P_\alpha).$$
First two numbers in both rows equal to the multiplicity of $M_\alpha$ in $M$ thanks to Proposition~\ref{Pss}, Lemma~\ref{Lmul} and Lemma~\ref{Ltra-}. 
\end{proof}

\begin{proposition}\label{Pmain} Module $\oplus_{\alpha\in\Pi}P_\alpha$ is a faithfully projective object of $\mathcal C^0$, see~\cite[Chapter~II]{B}. In particular, $\mathcal C^0$ is equivalent to the category of left finite-dimensional $\operatorname{End}_{\Ue(\frak g)}(\oplus_{\alpha\in\Pi}P_\alpha)$-modules. 
\end{proposition}
\begin{proof}All objects of $\mathcal C^0$ are of finite length and it is clear that $\operatorname{Hom}_{\mathcal C^0}(\oplus_{\alpha\in\Pi}P_\alpha, \cdot)$ is an exact functor which preserves arbitrary coproducts, i.e. direct sums in $\mathcal C^0$. This together with~\cite[Chapter~II,~Theorem~1.3]{B} implies that it is enough to show that $\oplus_{\alpha\in\Pi}P_\alpha$ is faithful, i.e. that for any object $M$ of $\mathcal C^0$ we have
$$\operatorname{Hom}_{\Ue(\frak g)}(\oplus_{\alpha\in\Pi}P_\alpha, M)\ne0.$$
This follows from Corollary~\ref{Cmain}.
\end{proof}
Put $A(\frak g):=\operatorname{End}_{\Ue(\frak g)}(\oplus_{\alpha\in\Pi}P_\alpha)$.
\subsection{Basis-dependent description of $A(\frak g)$.}
The goal of this subsection is to provide a convenient basis for the algebra $A(\frak g)$, see Theorem~\ref{Talg}. To do this we work out a complete set of linearly independent elements together with a multiplication rules for the elements of this set. In short, we have $4r-2$ basis elements where $r$ is the rank of $\frak g$, and the product of two elements of the basis is equal to either 0 or an element of the basis. It can be easily seen from this presentation that $A(\frak g)\cong A(\Gamma)$ where $\Gamma$ is a Dynkin diagram of $\frak g$ and $A(\Gamma)$ is a zigzag algebra attached to $\Gamma$, see~\cite{HK}. Alltogether this will complete the proof of Theorem~\ref{T1}b).
\begin{theorem}\label{Talg}There exists a basis $\underline \pi_\alpha, \underline \pi_\alpha^0~(\alpha\in\Pi)$, $\underline \varphi_{\alpha\beta}~(\alpha, \beta\in\Pi, \alpha$ and $\beta$ are adjacent) of $A(\frak g)$ such that

$$\underline \pi_\alpha \underline\pi_\beta=\begin{cases}\underline\pi_\alpha\hspace{10pt}~\mbox{~if~}\alpha=\beta\\0\hspace{10pt}\mbox{otherwise}\end{cases},\hspace{10pt}\underline \pi_\alpha \underline\pi_\beta^0=\underline \pi_\beta^0 \underline\pi_\alpha=\begin{cases}\underline\pi_\beta^0\hspace{10pt}~\mbox{~if~}\alpha=\beta\\0\hspace{10pt}\mbox{otherwise}\end{cases},$$
$$\underline \pi_\alpha \underline\varphi_{\beta\gamma}=\begin{cases}\underline\varphi_{\beta\gamma}\hspace{10pt}~\mbox{~if~}\alpha=\gamma\\0\hspace{10pt}\mbox{otherwise}\end{cases},\hspace{10pt}\underline \varphi_{\beta\gamma}\underline\pi_\alpha=\begin{cases}\underline\varphi_{\beta\gamma}\hspace{10pt}~\mbox{~if~}\alpha=\beta\\0\hspace{10pt}\mbox{otherwise}\end{cases}(\beta\mbox{~and~}\gamma\mbox{~are~adjacent}),$$
$$\underline \varphi_{\gamma\tau}\underline\varphi_{\alpha\beta}=\begin{cases}\underline\pi_{\alpha}^0\hspace{10pt}~\mbox{~if~}\alpha=\tau\mbox{~and~}\beta=\gamma\\0\hspace{10pt}\mbox{otherwise}\end{cases}~(\begin{tabular}{c}$\alpha$ and $\beta$ are adjacent\\$\gamma$ and $\tau$ are adjacent\end{tabular}),$$
$$\underline \pi_\alpha^0 \underline\pi_{\beta}^0=\underline \pi_\alpha^0 \underline\varphi_{\beta\gamma}=\underline\varphi_{\beta\gamma}\underline \pi_\alpha^0=0\hspace{10pt}(\beta\mbox{~and~}\gamma\mbox{~are~adjacent})$$
for all $\alpha, \beta, \gamma, \tau\in\Pi$.
\end{theorem}

To start with we compute $\dim \operatorname{Hom}_{\Ue(\frak g)}(P_\alpha, P_\beta)$.
\begin{lemma}\label{Ldimh} For $\alpha, \beta\in\Pi$ we have

a) $\dim \operatorname{Hom}_{\Ue(\frak g)}(P_\alpha, P_\alpha)=2$,

b) $\dim \operatorname{Hom}_{\Ue(\frak g)}(P_\alpha, P_\beta)=1$, if $\alpha$ and $\beta$ are adjacent,

c) $\dim \operatorname{Hom}_{\Ue(\frak g)}(P_\alpha, P_\beta)=0$, if $\alpha\ne\beta$ are not adjacent.\end{lemma}
\begin{proof}Is implied by Corollary~\ref{Cmain}, Lemma~\ref{Lm}, Lemma~\ref{Lp3}.\end{proof}
Next, we fix an element $\varphi_{\alpha\beta}\in \operatorname{Hom}_{\Ue(\frak g)}(P_\alpha, P_\beta)$ for the pair of adjacent roots $\alpha, \beta$ (this element is unique up to scaling). We have that $\dim \operatorname{Hom}_{\Ue(\frak g)}(M_\alpha, P_\alpha)=\dim \operatorname{Hom}_{\Ue(\frak g)}(P_\alpha, M_\alpha)=1$. Thus the composition of nonzero morphisms$$P_\alpha\to M_\alpha\to P_\alpha$$is unique up to scaling and we denote by $\pi^0_\alpha$ one such a composition. We denote by $\pi_\alpha$ the identity morphism on $P_\alpha$. 

The following proposition is a first approximation to Theorem~\ref{Talg}.
\begin{proposition}\label{Pbas}The elements $\pi_\alpha, \pi_\alpha^0~(\alpha\in\Pi)$ and $\varphi_{\alpha\beta}~(\alpha, \beta\in\Pi$, $\alpha$ and $\beta$ are adjacent) are a basis of $A(\frak g)$.\end{proposition}
\begin{proof} It is clear that $\pi_\alpha, \pi_\alpha^0$ are not proportional. This together with Lemma~\ref{Ldimh}a) implies that they form a basis of $\operatorname{Hom}_{\Ue(\frak g)}(P_\alpha, P_\alpha)$. Next, Lemma~\ref{Ldimh}b) implies that $\varphi_{\alpha\beta}$ is a basis of $\operatorname{Hom}_{\Ue(\frak g)}(P_\alpha, P_\beta)$ for a pair of adjacent roots $\alpha, \beta\in\Pi$. We left to mention that $\dim \operatorname{Hom}_{\Ue(\frak g)}(P_\alpha, P_\beta)=0$ if $\alpha$ and $\beta$ are not adjacent according to Lemma~\ref{Ldimh}c).\end{proof}
To complete the proof of Theorem~\ref{Talg} we need to evaluate the products of the elements of the basis.
\begin{lemma}\label{Lend}We have $(\pi_\alpha^0)^2=0$.\end{lemma}
\begin{proof}  Corollary~\ref{Cmain} implies that $P_\alpha$ has the unique proper maximal submodule which is precisely the kernel of $\pi_\alpha^0$. This implies that $P_\alpha$ is indecomposable and hence all endomorphisms of $P_\alpha$ are either scalar of nilpotent. The image of $\pi_\alpha^0$ is a simple submodule of $P_\alpha$ and therefore $\pi_\alpha^0$ is not a scalar operator on $P_\alpha$. Therefore $\pi_\alpha^0$ is nilpotent, and the fact that the image of $\pi_\alpha^0$ is simple implies that $(\pi_\alpha^0)^2=0$.\end{proof}
\begin{lemma}\label{Lpra}Let $\alpha, \beta\in\Pi$ be adjacent one to each other. Then

a) $\pi_\beta^0\varphi_{\alpha\beta}=0$,

b) $\varphi_{\alpha\beta}\pi_\alpha^0=0$,

c) $\varphi_{\alpha\beta}\varphi_{\beta\alpha}$ is proportional to $\pi_\beta^0$.\end{lemma}
\begin{proof}Part a). Corollary~\ref{Cmain} implies that $P_\beta$ has the unique proper maximal submodule $(P_\beta)_{sub}$ which is precisely the kernel of $\pi_\beta^0$. Let $P_{\alpha\beta}$ be the image of $\varphi_{\alpha\beta}$. Either $P_{\alpha\beta}\subset(P_\beta)_{sub}$ or $P_{\alpha\beta}=P_\beta$. In the first case we have $\pi_\beta^0\varphi_{\alpha\beta}=0$. Thus we proceed to the second case.

Assume that $P_{\alpha\beta}=P_\beta$. Then the fact that $P_\beta$ is projective implies that $$1=\dim \operatorname{Hom}_{\Ue(\frak g)}(P_\beta, P_\alpha)\ge\dim \operatorname{Hom}_{\Ue(\frak g)}(P_\beta, P_\beta)=2.$$ This is a contradiction.

Part b). Corollary~\ref{Cmain} implies that $P_\alpha$ has the unique simple submodule which is precisely the image of $\pi_\alpha^0$ and which is isomorphic to $M_\alpha$. We have $\dim \operatorname{Hom}_{\Ue(\frak g)}(M_\alpha, P_\beta)=\dim \operatorname{Hom}_{\Ue(\frak g)}(P_\beta, M_\alpha)=0$.

Part c). First, we show that $\varphi_{\alpha\beta}\varphi_{\beta\alpha}\ne0$. Assume to the contrary that $\varphi_{\alpha\beta}\varphi_{\beta\alpha}=0$. This means that  the image $P_{\beta\alpha}$ of $\varphi_{\beta\alpha}$ belongs to the kernel $\operatorname{Ker}\varphi_{\alpha\beta}$ of $\varphi_{\alpha\beta}$. Then we have a nonzero morphism
$$P_\alpha/P_{\beta\alpha}\to P_\beta.$$This immediately implies that the Jordan-H\"older multiplicities of $M_\beta$ in $P_{\beta\alpha}$  and $P_\alpha/P_{\beta\alpha}$ both are nonzero. Hence the Jordan-H\"older multiplicity of $M_\beta$ in $P_\alpha$ is at least 2. This is not the case.

Proposition~\ref{Pbas} implies that$$\varphi_{\alpha\beta}\varphi_{\beta\alpha}=c_1\pi_\beta+c_2\pi^0_\beta.$$
If $c_1\ne0$ then the fact that $(\pi_\beta^0)^2=0$ implies that $\varphi_{\beta\alpha}\varphi_{\alpha\beta}$ is surjective and therefore that $\varphi_{\beta\alpha}$ is surjective. This is wrong, see case a). Therefore $\varphi_{\alpha\beta}\varphi_{\beta\alpha}=c_2\pi_\beta^0$.\end{proof}
\begin{lemma}\label{Lpr3}Let $\alpha, \beta, \gamma\in\Pi$ be such that $\alpha$ and $\beta$ are adjacent, $\beta$ and $\gamma$ are adjacent, $\alpha\ne\gamma$. Then $\varphi_{\beta\gamma}\varphi_{\alpha\beta}=0$.\end{lemma}
\begin{proof}It is clear that $\alpha$ is not adjacent to $\gamma$. Thus $\dim\operatorname{Hom}_{\Ue(\frak g)}(P_\alpha, P_\gamma)=0$. Hence $\varphi_{\beta\gamma}\varphi_{\alpha\beta}=0$.\end{proof}
\begin{proof}[Proof of Theorem~\ref{Talg}]Lemma~\ref{Lpra} implies that, for a pair of adjacent roots $\alpha, \beta\in\Pi$, there exist a nonzero constant $c_{\alpha,\beta}$ such that $\varphi_{\beta\alpha}\varphi_{\alpha\beta}=c_{\alpha,\beta}\pi_\alpha^0$. Denote by $\alpha_0$ the unique simple root with 3 neighbours. For any $\alpha\in\Pi$ put $\alpha_0, \alpha_1,..., \alpha_{n(\alpha)}$ to be the shortest sequence of adjacent roots which connects $\alpha_0$ with $\alpha=\alpha_{n(\alpha)}$. Set
$$\underline\pi_{\alpha}:=\pi_\alpha (\forall\alpha),\hspace{10pt}\underline\varphi_{\alpha_{n(\alpha)-1}\alpha}:=\varphi_{\alpha_{n(\alpha)-1}\alpha}~(\alpha\ne\alpha_0)$$

$$\underline\pi_\alpha^0:=\begin{cases}(\prod\limits_{0\le i< n(\alpha)}\frac{c_{\alpha_{i+1}\alpha_i}}{c_{\alpha_i\alpha_{i+1}}})\pi_{\alpha}^0\hspace{10pt}\mbox{~if~}\alpha\ne\alpha_0\\\pi_{\alpha}^0\hspace{10pt}\mbox{otherwise}\end{cases},$$
$$\underline\varphi_{\alpha_{n(\alpha)}\alpha_{n(\alpha)-1}}:=\begin{cases}(\prod\limits_{0\le i<n(\alpha)}\frac{c_{\alpha_{i+1}\alpha_i}}{c_{\alpha_i\alpha_{i+1}}})\frac{\varphi_{\alpha_{n(\alpha)}\alpha_{n(\alpha)-1}}}{c_{\alpha_{n(\alpha)}\alpha_{n(\alpha)-1}}}\hspace{10pt}\mbox{~if~}\alpha\ne\alpha_0\\\mbox{undefined~otherwise}\end{cases}.$$
It can be easily checked through Lemmas~\ref{Lend},~\ref{Lpra},~\ref{Lpr3} that $\underline \pi_\alpha, \underline\pi_\alpha^0, \underline\varphi_{\alpha\beta}$ satisfy all conditions of Theorem~\ref{Talg}.\end{proof}
\section{Proof of Theorem~\ref{T1nint}}\label{Snint}
We need to expand our notation on Weyl groups, root systems e.t.c. to the nonintegral case. We use notation of Section~\ref{Scsfd}. Fix $\lambda\in\frak h^*$. Put $$\Delta^{\mathbb Z}:=\{\alpha\in\Delta\mid \frac{2(\lambda+\rho, \alpha)}{(\alpha, \alpha)}\in\mathbb Z\},
$$
It is clear that $\Delta_\lambda\subset\Delta^\mathbb Z$. We denote by $W^\mathbb Z$ the subgroup of $W$ generated by the reflections with respect to elements of $\Delta^\mathbb Z$. To the root system $\Delta^\mathbb Z$ we attach a Lie algebra $\frak g^\mathbb Z$ in such a way that $\frak h$ is a Cartan subalgebra of $\frak g^\mathbb Z$. We denote by $\frak b^\mathbb Z$ the Borel subalgebra of $\frak g^\mathbb Z$ attached to $\Delta^\mathbb Z\cap \Delta^+$. This allows as to define $M^\mathbb Z(\lambda), L^\mathbb Z(\lambda),  I^\mathbb Z(\lambda)$, and $O^\mathbb Z(\lambda)$.

To any root $\alpha\in\Delta$ we attach a coroot $\alpha^\vee:=\frac{2\alpha}{(\alpha, \alpha)}\in\frak h^*$. The set of coroots $\Delta^\vee$ is a root system, and the Lie algebra $\frak g^\vee$ attached to $\Delta^\vee$ is called {\it Langlands dual} to $\frak g$. It can be easily seen that $\frak g^\mathbb Z$ is not necessarily isomorphic to a subalgebra of $\frak g$, but $(\frak g^\mathbb Z)^\vee$ is canonically isomorphic to the subalgebra of $\frak g^\vee$ (it is defined as the sum of $\frak h$ with the weight spaces of weights of $(\Delta^\mathbb Z)^\vee$). It is worth to mention that if $\frak g$ is simple then $\frak g\cong\frak g^\vee$ if and only if $\frak g$ is not of type $B_n, C_n~(n\ge3)$.

\begin{definition}We say that a subalgebra $\frak k$ of $\frak g$ is an {\it r-subalgebra} if $\frak k$ and $\frak g$ have the same rank. By definition, $(\frak g^\mathbb Z)^\vee$ is an r-subalgebra of $(\frak g)^\vee$.\end{definition}

We denote by $\operatorname{PrId}^\lambda$ the set primitive ideals $I$ of $\operatorname{U}(\frak g)$ such that $I\cap\operatorname{Z}(\frak g)=m_\lambda$. The following theorem is a slight modification of~\cite[Theorem~2.5]{BV2}.
\begin{theorem}\label{Tvo} The map
$$W^\mathbb Z\to\operatorname{PrId}^\lambda\hspace{10pt}(w\to I(w(\lambda+\rho)-\rho))$$
is surjective.
\end{theorem}
Put $\rho^\mathbb Z:=\frac12\sum_{\alpha\in\Delta^\mathbb Z\cap\Delta^+}\alpha$. We say that $\lambda$ is {\it dominant with respect to $\Delta^\mathbb Z\cap\Delta^+$} if $(\lambda, \alpha^\vee)\in\mathbb Z_{\ge0}$ for all $\alpha\in\Delta^\mathbb Z\cap\Delta^+$. The following statement is a straightforward corollary of~\cite[Theorem~4.8, Proposition~2.28]{BV2}.
\begin{theorem}\label{Tired} We have
$$\dim\frak g-\dim O(\lambda)=\dim \frak g^\mathbb Z-\dim O^{\mathbb Z}(\lambda+\rho-\rho^\mathbb Z).$$\end{theorem}
This theorem allows us to expand Theorem~\ref{T1} beyond the integral case. Recall that $$\operatorname{PrId}^\lambda(O)=\{I\in\operatorname{PrId}^\lambda\mid \operatorname{Var}(I)=\bar O\}.$$
Proposition~\ref{Pt1} is implied by~ Proposition~\ref{Pprw} and the following statement.
\begin{proposition} Assume that $\frak g$ is simply-laced and not of type $A_n$. If $\lambda$ is nonintegral then $\operatorname{PrId}^\lambda(O)$ is empty.\end{proposition}
\begin{proof} Assume to the contrary that $\lambda$ is nonintegral, and there exists an ideal $I$ in $\operatorname{PrId}^\lambda(O)$. The assumption on $\lambda$ implies that $$\dim \frak g^\mathbb Z<\dim\frak g.$$
Theorem~\ref{Tired} implies that
$$\dim O=\dim\frak g-\dim \frak g^\mathbb Z+\dim O^{\mathbb Z}(\lambda+\rho-\rho^\mathbb Z),$$
and thus
$$\dim O\ge\dim\frak g-\dim\frak g^\mathbb Z.$$
The dimensions of $\dim O$ are written in Table~\ref{Tdimo} (they are evaluated through~\cite[Lemma~4.3.5]{CM}).
\begin{table}[h!]\caption{}\label{Tdimo}$\begin{array}{|c|c|c|c|c|c|c|c|c|}\hline A_n~(n\ge2)&B_n~(n\ge2)&C_n~(n\ge2)&D_n~(n\ge4)&E_6&E_7&E_8&F_4&G_2\\
\hline 2n&4n-4&2n&4n-6&22&34&58&16&6\\\hline\end{array}$\end{table}

Comparing these numbers with Proposition~\ref{Pmax} of Appendix applied to $(\frak g^\mathbb Z)^\vee\subset(\frak g)^\vee$ we see that
$$\dim O<\dim(\frak g)^\vee-\dim(\frak g^\mathbb Z)^\vee$$
due to the fact that $\frak g$ is simply-laced and not of type $A$. This is a contradiction.\end{proof}

We wish to mention that cases  $C_n, G_2$ were considered in~\cite[Corollary~7.1]{Pr1}. Theorem~\ref{T1nint} is implied by the following proposition.
\begin{proposition}\label{Pdesc} Assume that $\frak g$ is not simply-laced and is simple. Then the following statements are equivalent

a) $\operatorname{U}(\frak g, e)-f.d.mod^{m_\lambda}$ has a nonzero object,

b) $\operatorname{U}(\frak g, e)-f.d.mod^{m_\lambda}$ is semisimple with a unique simple object.\end{proposition}
\begin{proof}[Proof of Proposition~\ref{Pdesc}] It is enough to show that a) implies b). From now on we assume a). We need the following lemmas.
\begin{lemma}\label{Lmaxi1} Assume that $O(\lambda)=O$. Then

a) $(\frak g^\mathbb Z)^\vee$ is a maximal proper r-subalgebra of $\frak g^\vee$,

b) $\lambda+\rho-\rho^\mathbb Z$ is dominant with respect to $\Delta^\mathbb Z\cap\Delta^+$,

c) $W_{\lambda}=\{e\}$.\end{lemma}
\begin{proof}The minimal nilpotent orbit $O$ is not special under the assumption that $\frak g$ is not simply-laced, see~\cite{CM}. Thus $(\frak g^\mathbb Z)^\vee$ is a proper r-subalgebra of $\frak g^\vee$. Theorem~\ref{Tired} implies that
$$\dim O=\dim\frak g-\dim \frak g^\mathbb Z+\dim O^{\mathbb Z}(\lambda+\rho-\rho^\mathbb Z),$$
and thus that
$$\dim O\ge\dim\frak g-\dim\frak g^\mathbb Z.$$
The dimensions of $\dim O$ are written in Table~\ref{Tdimo}. Comparing these numbers with Proposition~\ref{Pmax} of Appendix applied to $(\frak g^\mathbb Z)^\vee\subset(\frak g)^\vee$ we see that
$$\dim O\le \dim(\frak g)^\vee-\dim(\frak g^\mathbb Z)^\vee$$
due to the fact that $\frak g$ is not simply-laced. Therefore Proposition~\ref{Pmax} and the fact that $(\frak g^\mathbb Z)^\vee$ is an r-subalgebra of $\frak g^\vee$, implies that $(\frak g^\mathbb Z)^\vee$ is a maximal r-subalgebra of $\frak g^\vee$. Also we have that $$\dim O^\mathbb Z(\lambda+\rho-\rho^\mathbb Z)=0,$$
and thus that $\lambda+\rho-\rho^\mathbb Z$ is dominant with respect to $\Delta^\mathbb Z$. The latter condition implies that $W_{\lambda}$ is trivial. This completes the proof.\end{proof}

Lemma~\ref{Lmaxi1} together with Proposition~\ref{Pdesc}a) implies a), b), c) of Lemma~\ref{Lmaxi1}. The lemma below is well known, and we provide the proof of it only for the convenience of a reader.
\begin{lemma}\label{Ldelta} Let $\Delta_1, \Delta_2\subset\Delta$ be subsets of $\Delta$ such that $$\frak g_1:=\frak h\oplus_{\alpha\in\Delta_1}\frak g_\alpha,\hspace{10pt}\frak g_2:=\frak h\oplus_{\alpha\in\Delta_1}\frak g_\alpha$$ are r-subalgebras of $\frak g$. Then the following conditions are equivalent

a) $\frak g_1$ is conjugate to $\frak g_2$ by the adjoint group $\operatorname{Adj}(\frak g)$ of $\frak g$,

b) there exists $w\in W$ such that $w(\Delta_1)=\Delta_2$.
\end{lemma}
\begin{proof} First, we show that a) implies b). We fix $g\in\operatorname{Adj}(\frak g)$ such that $g(\frak g_1)=\frak g_2$. By definition, $\frak h$ is Cartan subalgebra of both $\frak g_1$ and $\frak g_2$. Thus $g(\frak h)$ is a Cartan subalgebra of $\frak g_2$. Therefore there exists an element $g_2$ of the adjoint group $\operatorname{Adj}(\frak g_2)$ of $\frak g_2$ such that $g_2(g(\frak h))$. Put $g':=g_2g$. We have
$$g'(\frak g_1)=\frak g_2,\hspace{10pt}g'(\frak h)=\frak h.$$
The second condition implies that $g'$ can be represented by an element $w\in W$ and the first condition implies that this $w$ is as desired in b).

Next, we show that b) implies a). The Weyl group can be identified with the quotient of the normalizer of $\frak h$ in $\operatorname{Adj}(\frak g)$ by the centralizer of $\frak h$ in $\operatorname{Adj}(\frak g)$. Thus there exists $g\in \operatorname{Adj}(\frak g)$ such that
$$g(\frak h)=\frak h,\hspace{10pt}g(\frak g_\alpha)=\frak g_{w(\alpha)}.$$
It is clear that $g(\frak g_1)=\frak g_2$.\end{proof}

\begin{lemma}\label{Lmod1} Assume that $\operatorname{U}(\frak g, e)-f.d.mod^{m_\lambda}$ has a nonzero object. Then $\operatorname{U}(\frak g, e)-f.d.mod^{m_\lambda}$ has a unique simple object $M(\lambda)$. Moreover, there exists $w\in W^\mathbb Z$ such that $M_{f.d.}(\lambda)=\mathcal P^{-1}(I(w(\lambda+\rho)-\rho))$.\end{lemma}
\begin{proof}Theorem~\ref{Tvo} together with Proposition~\ref{Pprw} implies that any simple object of $\operatorname{U}(\frak g, e)-f.d.mod^{m_\lambda}$ is of the form $\mathcal P^{-1}(I(w(\lambda+\rho)-\rho))$ for $w\in W^\mathbb Z$.

We consider $w\in W^\mathbb Z$ such that $\mathcal P^{-1}(I(w(\lambda+\rho)-\rho))$ belongs to $\operatorname{U}(\frak g, e)-f.d.mod^{m_\lambda}$. The latter condition is equivalent to $O(w(\lambda+\rho)-\rho)=O$. This implies that $w(\lambda+\rho)-\rho^\mathbb Z$ is dominant with respect to $\Delta^\mathbb Z\cap\Delta^+$. This implies that $w$ is unique. Thus we have shown that $\operatorname{U}(\frak g, e)-f.d.mod^{m_\lambda}$ has a unique simple object. We denote this object $M_{f.d.}(\lambda)$.
\end{proof}
We left to prove that $\operatorname{Ext}^1(M_{f.d.}(\lambda), M_{f.d.}(\lambda))$. If $\frak g\cong G_2$ or $\frak g\cong C_n$ then the statement of Proposition~\ref{Pdesc} follows immediately from~\cite[Corollary~7.1]{Pr1}. From now on we assume that $\frak g$ is not of type $C_n, G_2$, i.e. that $\frak g$ is of type $B_n, F_4$. We use notation for roots and weights of~\cite{Bou}, and consider both cases simultaneously. The underlying calculations can be easily done through~\cite{Bou}. 

Set $\mu_0$ as in~\cite[Table~3]{JoM}. Put $$\Delta(\frak k):=\{\alpha\in\Delta\mid (\alpha^\vee, \mu_0)\in\mathbb Z\},\hspace{10pt}\frak k^\vee:=\frak h\bigoplus\limits_{\alpha\in\Delta(\frak k)}(\frak g)^\vee_{\alpha^\vee}.$$
Dimension arguments together with Proposition~\ref{Pmax} implies that $\frak k^\vee$ is a maximal r-subalgebra of $\frak g^\vee$. We have that $(\frak g^\mathbb Z)^\vee$ is also a maximal r-subalgebra of $\frak g^\vee$. Lemma~\ref{Ldelta} implies that there exists $w'\in W$ such that $w'(\Delta^\mathbb Z)=\Delta(\frak k)$. Put $\lambda':=w'(\lambda+\rho)-\rho$. We have that $\operatorname{U}(\frak g, e)-f.d.mod^{m_\lambda}=\operatorname{U}(\frak g, e)-f.d.mod^{m_{\lambda'}}$. Thus we can assume that $\lambda=\lambda'$. This immediately implies that $\Delta^\mathbb Z=\Delta(\frak k)$.


Let $\gamma\in\Pi$ be the unique simple root such that $(\gamma, \chi)\ne 0$ (in both cases $\gamma=\alpha_1$). Then $\Pi\backslash\gamma\in\Delta^\mathbb Z$.

It is easy to check that $2\gamma^\vee$ belongs to the lattice generated by $(\Delta^\mathbb Z)^\vee$. Thus
$$2(\gamma^\vee, \lambda)\in\mathbb Z.$$
Therefore either $(\gamma^\vee, \lambda)\in\mathbb Z$ or $(\gamma^\vee, \lambda)\in\frac12+\mathbb Z$. We have that $(\alpha^\vee, \lambda)\in\mathbb Z$ for all $\alpha\in(\Pi\backslash\gamma)\subset\Delta^\mathbb Z$. Thus if $(\gamma^\vee, \lambda)\in\mathbb Z$ then $\Pi\subset \Delta^\mathbb Z$, and hence $\Delta^\mathbb Z=\Delta$. This is a contradiction.

Thus $(\gamma^\vee, \lambda)\in\frac12+\mathbb Z$. The fact that $(\alpha^\vee, \lambda)\in\mathbb Z$ for all $\alpha\in(\Pi\backslash\gamma)\subset\Delta^\mathbb Z$ implies that
$$(\lambda-\mu_0, \alpha^\vee)\in\mathbb Z.$$
As a consequence $\lambda-\mu_0\in\Lambda, \lambda-\mu_0+\rho\in\Lambda$. Lemma~\ref{Lmaxi1} implies that $\lambda+\rho-\rho^\mathbb Z$ is dominant with respect to $\Delta^\mathbb Z\cap\Delta^+$. It is straightforward to check that $\mu_0-\rho^\mathbb Z$ is dominant with respect to $\Delta^\mathbb Z\cap\Delta^+$.

To finish the proof we need the following lemma.
\begin{lemma}\label{Lmod2}Assume that

a) both $\lambda+\rho-\rho^\mathbb Z$ and $\mu_0-\rho^\mathbb Z$ are dominant with respect to $\Delta^\mathbb Z\cap\Delta^+$,

b) $\lambda-\mu_0+\rho\in\Lambda$,

c) $W_\lambda=W_{\mu_0}$.\\
Then $\operatorname{U}(\frak g, e)-f.d.mod^{m_\lambda}$ is equivalent to $\operatorname{U}(\frak g, e)-f.d.mod^{m_{\mu_0-\rho}}$.\end{lemma}
\begin{proof} We recall that $\operatorname{U}(\frak g, e)-f.d.mod^{m_\lambda}$ is equivalent to a category of $\frak g$-modules $\mathcal C^\lambda$, see Section~\ref{Spr}. Conditions a), b), c) implies together with~\cite[Theorem, Section 4.1]{BGG} that $\mathcal C^\lambda$ is equivalent to $\mathcal C^{\mu_0-\rho}$. Hence $\operatorname{U}(\frak g, e)-f.d.mod^{m_\lambda}$ is equivalent to $\operatorname{U}(\frak g, e)-f.d.mod^{m_{\mu_0-\rho}}$.\end{proof}
Thanks to Lemmas~\ref{Lmod1},~\ref{Lmod2} we have that
$$\dim\operatorname{Ext}^1(M_{f.d.}(\lambda),M_{f.d.}(\lambda))=\dim\operatorname{Ext}^1(M_{f.d.}(\mu_0-\rho), M_{f.d.}(\mu_0-\rho)).$$
Thus we left to show that $\dim\operatorname{Ext}^1(M ,M)=0$ for one nonzero simple finite-dimensional $\operatorname{U}(\frak g, e)$-module $M$. This was done in Proposition~\ref{Pext}.\end{proof}
\section{Acknowledgements}
I would like to thank Alexander Premet for many useful comments on this article, and many stimulating questions on the previous versions of it. I would like to thank Tobias Kildetoft for the reference~\cite{Dou}, Catharina Stroppel for the reference~\cite{HK}, Lewis Topley for the reference~\cite{LosQ}, \'Ernest Vinberg for the reference~\cite{BS}.
\section{Appendix: Maximal reductive root subalgebras}
Let $\frak g$ be a reductive Lie algebra. By definition, the rank of $\frak g$ equals the dimension of a Cartan subalgebra of $\frak g$. We say that a subalgebra $\frak k$ of $\frak g$ is an {\it r-subalgebra} if the rank of $\frak k$ equals the rank of $\frak g$. The description of maximal {\it by inclusion} r-subalgebras of simple Lie algebras are very well known, see~\cite{BS}, see also~\cite{Bou},~\cite{Dyn}. We used a description of maximal {\it by dimension} r-subalgebras in Section~\ref{Snint}. Of course, the second one is a straightforward consequence of the first one, but this consequence requires plenty of not very conceptual computations. We decided to provide these computations in this Appendix.

\begin{proposition}\label{Pmax} Let $\frak g$ be a simple Lie algebra. Then the conjugacy classes of maximal proper r-subalgebras $\frak k$ of $\frak g$ such are listed in Table~\ref{Tab1}. In particular, such a conjugacy class is unique if $\frak g\not\cong D_4$. 
\begin{table}[h!]\caption{}\label{Tab1}
$\begin{array}{|c|c|c|c|}\hline \frak g&\frak k&\operatorname{codim}(\frak g/\frak k)&\\\hline
A_{l}&A_{l-1}\oplus\mathbb F&2l-2&l\ge1\\
B_l&D_l&2l&l\ge 2\\
C_l&C_{l-1}\oplus C_1&4l-4&l\ge2\\
D_l&D_{l-1}\oplus\mathbb F&4l-4&l\ge5\\
D_4&D_3\oplus\mathbb F, A_3\oplus \mathbb F, A_3'\oplus\mathbb F&12&\\
E_6&D_5\oplus\mathbb F&32&\\
E_7&E_6\oplus\mathbb F&54&\\
E_8&E_7\oplus A_1&112&\\
F_4&B_4&16&\\
G_2&A_2&6&\\\hline
\end{array}.$\end{table}
\end{proposition}
\begin{proof} The conjugacy classes of such subalgebras can be identified with the conjugacy classes of subgroups of maximal rank of the compact group attached to the root system of $\frak g$. The latter conjugacy classes are described in~\cite[Table]{BS}. We reproduce the needed part of~\cite[Table]{BS} below in Table~\ref{T3}.
\begin{table}[h!]\caption{}\label{T3}$\begin{array}{|c|c|c|}\hline \frak g&\frak m&\\
\hline A_l & A_i\oplus A_{l-i-1}\oplus T&1\le i< l, l\ge1\\
\hline B_l & D_l&l\ge 2\\
& B_i\oplus D_{l-i}&1\le i< l, l\ge 2\\
& B_{l-1}\oplus T&j\ge 2\\
\hline C_l& C_i\oplus C_{l-i}& 1\le i< l, l\ge 2\\
& A_{l-1}\oplus T&l\ge 2\\
\hline D_l& D_i\oplus D_{l-i}& 1\le i\le l, l\ge 4\\
& A_{l-1}\oplus T&l\ge 4\\
& D_{l-1}\oplus T&l\ge 4\\
\hline E_6&A_1\oplus A_5, A_2\oplus A_2\oplus A_2&\\
&D_5\oplus\mathbb F&\\
\hline E_7& A_1\oplus D_6, A_7, A_2\oplus A_5&\\
&E_6\oplus\mathbb F&\\
\hline E_8&D_8, A_1\oplus E_7, A_8&\\
&A_2\oplus E_6, A_4\oplus A_4&\\
\hline F_4& A_1\oplus C_3, B_4, A_2\oplus A_2&\\
\hline G_2& A_1\oplus A_1, A_2&\\
\hline \end{array}$\end{table}

We left to figure out which subalgebras of this list are maximal by dimension. For the exceptional groups this is very straightforward cause the list of maximal by inclusion subalgebras is finite. We proceed with the classical cases one-by-one.

Case A: we claim that $A_{l-1}\oplus \mathbb F$ has the maximal dimension in $A_l$, and thus that
$$\dim(A_{l-1}\oplus\mathbb F)>\dim (A_i\times A_{l-i-1}\oplus\mathbb F)$$
if $0<i<l-1$. This is equivalent to the following statements
$$\begin{array}{c}l^2>(i+1)^2+(l-i)^2-1,\\
2li-2i^2-2i\ge0,\\
2i(l-i-1)\ge0\end{array}(0<i<l-1).$$
The latter statement is trivial.

Case B: we claim that $D_{l}$ has the maximal dimension in $A_l$, and thus that
$$\begin{array}{ccc}(B1)&\dim D_l>\dim(B_i\oplus D_{l-i}),&~0<i<l, l\ge3\\
(B2)&\dim D_l>\dim(B_{l-1}\oplus\mathbb F),&l\ge3.\end{array}$$
We now verify statements (B1) and (B2).

Statement (B1) is equivalent to the following inequalities
$$\begin{array}{l}2l^2-l> (2i^2+i)+(2(l-i)^2-(l-i)),\\
2i(2l-i)-i> 2i^2+i,\\
2i(2l-2i-1)>0\\\end{array}(0<i<l).$$
The latter statement is trivial.

Statement (B2) is equivalent to the following inequalities
$$\begin{array}{l}2l^2-l> (2(l-1)^2+(l-1))+(2(l-i)^2-(l-i)),\\
2(l-1)>0\\\end{array}(l\ge2).$$
The latter statement is trivial.

Case C: we claim that $C_{l-1}\oplus C_1$ has the maximal dimension in $C_l$, and thus that
$$\begin{array}{ccc}(C1)&\dim(C_l\oplus C_1)>\dim(C_i\oplus C_{l-i}),&~1<i<l-1, l\ge2\\
(C2)&\dim(C_{l-1}\oplus C_1)>\dim(A_{l-1}\oplus\mathbb F),&l\ge2.\end{array}$$
We now verify statements (C1) and (C2).

Statement (C1) is equivalent to the following inequalities
$$\begin{array}{l}2(l-1)^2+(l-1)+3> (2i^2+i)+(2(l-i)^2-(l-i)),\\
2(i-1)(2l-i-1)-2(i-1)(i+1)>0,\\
4(i-1)(l-i-1)>0\\\end{array}(1<i<l-1).$$
The latter statement is trivial.

Statement (C2) is equivalent to the following inequalities
$$\begin{array}{l}2(l-1)^2+(l-1)+3> l^2,\\
(l-2)^2+l>0\\\end{array}(l\ge2).$$
The latter statement is trivial.

Case D: we claim that $D_{l-1}\oplus F$ has the maximal dimension in $D_l$, and thus that
$$\begin{array}{ccc}(D1)&\dim(D_{l-1}\oplus\mathbb F)>\dim(D_i\oplus D_{l-i}),&~1<i<l-1, l\ge4\\
(D2)&\dim(D_{l-1}\oplus\mathbb F)>\dim(A_{l-1}\oplus\mathbb F),&l\ge4.\end{array}$$
We now verify statements (D1) and (D2).

Statement (D1) is equivalent to the following inequalities
$$\begin{array}{l}2(l-1)^2-(l-1)+1> (2i^2-i)+(2(l-i)^2-(l-i)),\\
2(i-1)(2l-i-1)-(i-1)-2(i-1)(i+1)+(i-1)>0,\\\end{array}(1<i<l-1).$$
The latter statement is trivial.

Statement (D2) is equivalent to the following inequalities
$$\begin{array}{l}2(l-1)^2-(l-1)+1> l^2,\\
(l-1)(l-4)>0\\\end{array}(l\ge4).$$
If $l\ge 5$ the latter statement is trivial. If $l=4$ then $D_3=A_3$ and the subalgebras have the same dimension $D_3\oplus\mathbb F$ and $A_3\oplus\mathbb F, A_3'\oplus\mathbb F$ and are not conjugate in $D_4$.\end{proof}


\end{document}